\documentclass[12pt]{amsart}

\usepackage{amsfonts,amssymb,amsmath,amsthm,epsfig,euscript}

\newtheorem{theorem}{Theorem}
\newtheorem{lemma}[theorem]{Lemma}

\newtheorem{proposition}{Proposition}
\newtheorem{observation}{Observation}

\theoremstyle{definition}
{
\newtheorem{definition}{Definition}

}

\long\def\symbolfootnote[#1]#2{\begingroup
\def\thefootnote{\fnsymbol{footnote}}\footnote[#1]{#2}\endgroup}

\newcommand\p{\circle*{0.3}}

\newcommand{\red}{\mathrm{red}}

\newcommand{\sg}{\sigma}

\def\P{\mathbb{P}}

\newcommand{\fig}[2]{\begin{figure}[ht]
\centerline{\scalebox{.66}{\epsfig{file=#1.eps}}}
\caption{#2}
\label{fig:#1}
\end{figure}}

\title{Representing Graphs via Pattern Avoiding Words}

\author[M. Jones]{Miles Jones}
\thanks{Instituto de Matem\'atica,Universidad de Talca, Camino Lircay S/N Talca, Chile. Email: \texttt{mjones@inst-mat.utalca.cl} \ Supported by FONDECYT (Fondo Nacional de Desarrollo Cient\'{\i}fico y
Tecnol\'ogico de Chile) postdoctoral grant \#3130631.}

\author[S. Kitaev]{Sergey Kitaev}
\thanks{Department of Computer and Information Sciences, University of Strathclyde, Glasgow G1 1XH, UK. Email: \texttt{sergey.kitaev@cis.strath.ac.uk}}

\author[A. Pyatkin]{Artem Pyatkin}
\thanks{Sobolev Institute of Mathematics; Novosibirsk State University; Novosibirsk, Russia. Email: \texttt{artem@math.nsc.ru}}

\author[J. Remmel]{Jeffrey Remmel}
\thanks{Department of Mathematics, University of California, San Diego, La Jolla, CA 92093-0112. USA. Email: \texttt{jremmel@ucsd.edu}}

\begin{document}
\maketitle

\begin{abstract}
\noindent
The notion of a word-representable graph has been
studied in a series of papers in the literature. A graph $G=(V,E)$ is word-representable if there exists a word $w$ over the alphabet $V$ such that letters $x$ and $y$ alternate in $w$ if and only if $xy$ is an edge in~$E$.
If $V =\{1, \ldots, n\}$, this is equivalent to saying that
$G$ is word-representable if for all
$x,y \in \{1, \ldots, n\}$, $xy \in E$ if and only if
the subword $w_{\{x,y\}}$ of $w$ consisting of all occurrences
of $x$ or $y$ in $w$ has no consecutive occurrence of the pattern 11.

In this paper, we introduce the study of $u$-representable graphs for any word $u \in \{1,2\}^*$. A graph $G$ is $u$-representable if and only if
there is a labeled version of $G$, $G=(\{1, \ldots, n\}, E)$,
and a word $w \in \{1, \ldots, n\}^*$ such that for
all $x,y \in \{1, \ldots, n\}$, $xy \in E$ if and only if
$w_{\{x,y\}}$ has no consecutive
occurrence of the pattern $u$. Thus, word-representable
graphs are just $11$-representable graphs. We show
that for any $k \geq 3$, every finite graph $G$ is
$1^k$-representable. This contrasts with the
fact that not all graphs are 
11-representable graphs.

The main focus of the paper is the study of
$12$-representable graphs.
In particular, we classify the $12$-representable trees.
We show that any $12$-representable graph is a
comparability graph and the class of $12$-representable graphs include the classes of co-interval graphs and permutation graphs.
We also state a number of facts on $12$-representation of induced subgraphs of a grid graph. \\

\end{abstract}

\section{Introduction}

The notion of a {\em word-representable} graph was first defined
in \cite{KP2008}.  A graph $G=(V,E)$ is word-representable if there exists a word $w$ over the alphabet $V$ such that letters $x$ and $y$ alternate in $w$ if and only if $xy$ is an edge in $E$. For example, the cycle graph on four vertices labeled by 1, 2, 3 and 4 in clockwise direction can be represented by the word 14213243.  Word-representable graphs have been studied in a series of papers \cite{AKM}--\cite{HKP2011}, \cite{K}, \cite{KP2008}--\cite{KitSalSevUlf2}, and they will be the main subject of an up-coming book~\cite{KL}.

The first  examples of graphs
that are not word-representable were given in \cite{KP2008}. In fact, V. Limouzy [private communication, 2014] noticed that 
it is NP-hard to determine whether a given graph is word-representable, see \cite{KL} for the details. 
In \cite{KS} it was proved that any comparability graph $G$ is not just word-representable, but it is {\em permutationally} word-representable. That is, for the graph $G$, there exists a word $w$ with the necessary letter alternation properties such that $w$ is obtained by concatenating a number of permutations.

The key observation that motivated this paper
was the fact that the study of word-representable graphs
is naturally connected with the study of patterns in
words. That is, let $\P = \{1,2, \ldots \}$ be the set of positive integers and $\P^*$ be the set of all words over $\P$. If $n \in \P$, then
we let $[n] = \{1, \ldots, n\}$ and $[n]^*$ denote the set of
all words over $[n]$. Given a word $w =w_1 \ldots w_n$ in $\P^*$, we let
$A(w)$ be the set of letters occurring in $w$.  For example,
if $w = 4513113458$, then $A(w) = \{1,3,4,5,8\}$.  If
$B \subseteq A(w)$, then we let $w_B$ be the word that results
from $w$ by removing all the letters in $A(w) \setminus B$.  For
example, if  $w = 4513113458$, then $w_{\{1,3,5\}} =
5131135$. If $u \in \P^*$, we let
$\red(u)$ be the word that is obtained from $u$ by replacing
each occurrence of the $i$-th smallest letter that occurs in
$w$ by $i$. For example, if $u = 347439$, then
$\red(u) = 123214$.

Given a word $u =u_1 \ldots u_j \in \P^*$ such
that $\red(u) =u$,  we say
that a word $w = w_1 \ldots w_n \in \P^*$ has a {\em $u$-match starting at position $i$} if $\red(w_i w_{i+1} \ldots w_{i+j-1}) =u$.
Then
we can rephrase the definition of word-representable graphs
by saying that a graph $G$ is word-representable
if and only if there is a labeling $G=([n],E)$,  and
a word $w \in [n]^*$ such that for all $x,y \in [n]$,
$xy \in E$ if and only if $w_{\{x,y\}}$ has no $11$-matches.

This led us to the following defintion.
Given a word $u \in [2]^*$ such that $\red(u) =u$, we
say that a graph $G$ is {\em $u$-representable} if
and only if there is a labeling
$G=([n],E)$, and a word $w \in [n]^*$ such
that for all $x,y \in [n]$, $xy \in E$ if and only if
$w_{\{x,y\}}$ has no $u$-matches. 
In this case we say
that $w$ $u$-represents 
$G=([n],E)$.

This definition leads to a number of natural
questions.  For example, how much of the theory
of $11$-representable graphs carries over to
$u$-representable graphs? Can we
classify the $u$-representable graphs for
small words $u$ such as $u =111$, $u=1111$, $u=12$, or
$u =121$?  If a graph $G=([n],E)$ is $u$-representable, can
we always find a word $w$ which is a shuffle  of
a finite set of permutations representing some labeled version
of $G$?

Given how involved the theory of word-representable graphs is, our
first surprise was the fact that every graph is $111$-representable.  Indeed, we will show that for every $k \geq 3$,
every graph is $1^k$-representable. Thus, we decided to explore the
next simplest case, which is the class of $12$-representable
graphs. 

It turns out that there is a rich theory behind the class of
$12$-representable graphs. For example,
we will show that not every  graph is $12$-representable.
In particular,
the cycles $C_n$ for $n \geq 5$ are not
$12$-representable.  We will also show that there are non-$12$-representable trees, which contrasts with
the fact that every tree is $11$-representable.  In fact, we
will give a complete classification of the $12$-representable
trees.  We say that a tree $T =(V,E)$ is a {\em double caterpillar}
if and only if all vertices are within distance 2 of a central
path.  We will prove that a tree $T$ is $12$-representable
if and only if $T$ is a double caterpillar.

Further, we will show that the class of 12-representable
graphs is properly included in the class of comparability graphs, and it properly includes the classes of co-interval graphs and permutation graphs as shown in Figure~\ref{overview}. It turns out that the notion of $12$-representable graphs is a natural generalization of the notion of permutation graphs.

Figure~\ref{overview} also gives examples of graphs inside/outside the involved graph classes. For instance, even cycles of length at least 6, being comparability graphs, are not 12-representable (see Theorem~\ref{cycles-thm} in Section~\ref{12-matching-repr}); also, odd wheels on six or more vertices are not 11-representable \cite{KP2008}.

\begin{figure}[ht]
\begin{center}
\includegraphics[scale=0.55]{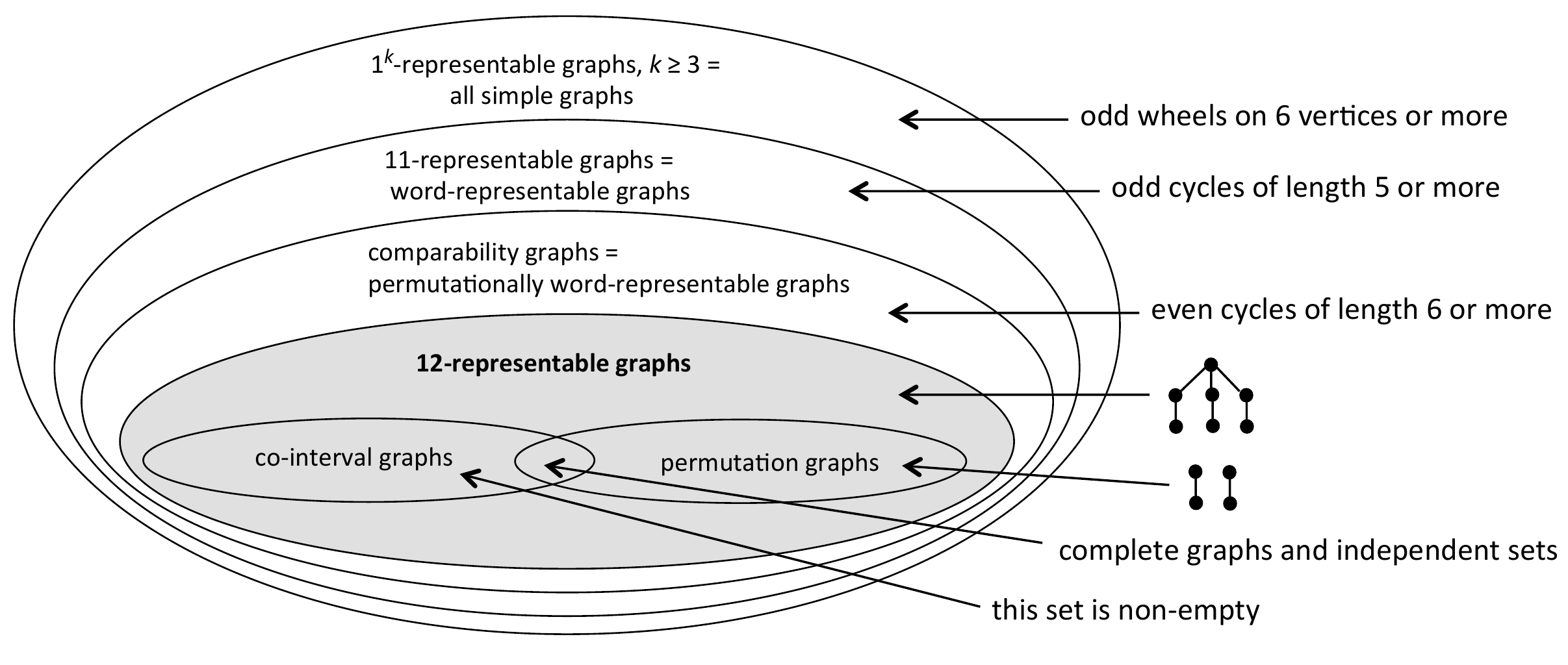}
\end{center}
\vspace{-10pt}
\caption{The place of 12-representable graphs in a hierarchy of graph classes.}
\label{overview}
\end{figure}

A fundamental difference between word-representable graphs and $u$-representable graphs is the following.  For word-representable graphs,
it is not so important whether we deal with labeled or unlabeled graphs: two isomorphic graphs are both either word-representable or not.
On the other hand, for certain $u$, a given graph $G$ may have
two different labeled versions $G_1 =([n], E_1)$ and
$G_2 =([n],E_2)$ such that there is a word
$w$ which $u$-represents $G_1$ but there is no word $w'$ which $u$-represents
$G_2$. In fact, we will see the phenomenon in the case where
$u=12$. This is why we say an unlabelled graph is  $u$-representable
if it admits labeling $G=([n],E)$ such that there is word
$w$ which $u$-represents $G$.

This paper is organized as follows.  In Section~\ref{preliminaries}, we give all necessary definitions and show that any graph is $1^k$-representable if $k \geq 3$. Some basic properties of $12$-representable graphs are established in Section~\ref{12-matching-repr}. In Section~\ref{trees}, we shall prove that a tree $T$ is $12$-representable if and only if $T$ is a double caterpillar.
In Section~\ref{graph-classes}, we compare the class of $12$-representable graphs to other graph classes thus explaining Figure~\ref{overview}. In Section~\ref{conclusion}, we provide a discussion of $12$-representability of induced subgraphs of a grid graph.
 Finally, in Section~\ref{other-notions}, we  introduce a number of new ways to define representability of simple graphs, directed graphs, and hypergraphs via
words subject to certain pattern avoidance conditions.
We also define several analogues of {\em Wilf-equivalencies} in Section~\ref{other-notions}, which are yet to be studied.

\section{Preliminaries}\label{preliminaries}

A simple graph $G =(V,E)$ consists of a set of vertices
$V$ and a set of edges $E$ of the form $xy$ where
$x,y \in V$ and $x \neq y$.  All graphs considered in this paper are simple and finite.
Sometimes $V$  will be a finite subset of $\P$. In this case we say that the graph is {\em labeled}. For an unlabeled graph we define its {\em labeling} as an assignment to its vertices some elements of $\P$ (labels).


If $G = (V,E)$ is a labeled graph (i.~e. $V \subset \P$) and $|V| =n$, then
the {\em reduction} of $G$, denoted
$\red(G)$, is a relabeling $G' = (\{1, \ldots,n\},E')$ such that
the label on the $i$-th smallest vertex
of $V$ is replaced by $i$.


Let $G=(V,E)$ be a graph and $v\in V$. Then we
say that the  graph $G'=(V',E')$ is obtained by adding a {\em copy} $v'$ of the vertex $v$ if $V'=V\cup\{v'\}$ and $E'=E\cup E^*$, where $v'a\in E^*$ if and only if $va \in E$
 and $E^*$ does not contain edges not involving $v'$.
%
%
%
%
%
%
%
%
%
%
%
%

Given a word $w =w_1 \ldots w_k \in \P^*$ and $x,y \in A(w)$, we say
that $x$ and $y$ alternate in $w$ if $w_{\{x,y\}}$ is either
of the form $xyxyxy \ldots$ of even or odd length or
$yxyx \ldots $ of even or odd length. Let $G =(V,E)$ be a
labeled graph.
Then we say that
$G$ is {\em word-representable} if there exists a
word $w \in V^*$ such that for all $x,y \in V$,
$xy$ is an edge in $E$ if and only if $x$ and $y$ alternate
in $w$.   In such a situation, we say that $w$ {\em word-represents} $G$ and
$w$ is called a {\em word-representant} of $G$.

We say that $H =(V',E')$ is an {\em induced subgraph} of $G =(V,E)$ if
$V' \subseteq V$ and for all $x,y \in V'$, $xy \in E'$
if and only if $xy \in E$.  Then we have the following
observation establishing the hereditary nature of the notion of graph word-representability.

\begin{observation}\label{subgraph}
If $G =(V,E)$ is word-representable and
$H=(V',E')$ is an induced subgraph of $G$, then $H$ is word-representable.
\end{observation}
Indeed, it is easy to see that
if $w$ represents $G=(V,E)$, then $w_{V'}$ represents
$H =(V',E')$.

In this paper, we introduce two generalizations
of the notion of a word-representable graph --- see Definitions~\ref{def-match-repr} and~\ref{exact-match-repr}. In Section~\ref{other-notions}, we will
discuss several other natural notions of representing
simple graphs, directed graphs, and hypergraphs by words
subject to certain pattern avoidance conditions.
The key to our main generalization is
to re-frame the notion of word-representable graphs in the language
of patterns in words. Note that $x$ and $y$ alternate in a
word $w \in \P^*$ if and only if $w_{\{x,y\}}$ has no $11$-match. Thus, a
graph $G =([n],E)$ is word-representable if
and only if there is a word $w \in [n]^*$ such that
for all $x,y \in [n]$, $xy$ is an edge in $E$ if and only
if $w_{\{x,y\}}$ has no $11$-match.  This leads us to our main
definition.

\begin{definition}\label{def-match-repr} Let $u = u_1 \ldots u_j$ be a word in $\{1,2\}^*$ such
that $\red(u) =u$.  Then we say that a labeled graph
$G =([n],E)$ is {\em $u$-representable}
if there is a word $w \in \P^*$ such that
for all $x,y \in [n]$, $xy  \in E$ if and only if
$w_{\{x,y\}}$ has no $u$-match. We say that an
 unlabeled graph $H$ is $u$-representable if there exits
a labeling of $H$, $H'=([n],E')$, such
that $H'$ is $u$-representable. In such a situation, we say that $H'$ {\em realizes the $u$-representability} of $H$.
\end{definition}
Thus, by Definition~\ref{def-match-repr}, $G$ is word-representable if and only if $G$ is
$11$-representable. Note that replacing ``word-representable graphs'' by ``$u$-representable graphs'' in Observation~\ref{subgraph}, we would obtain a true statement establishing the hereditary nature of $u$-representable graphs. The theory of word-representable graphs is rather involved, and thus the following theorem, where $1^k$ denotes $k$ 1s, came as a surprise to us.

\begin{theorem}\label{11111-matching} For every $k \geq 3$,
every finite graph $G$ is $1^k$-representable.
\end{theorem}

\begin{proof} Fix any $k\geq 3$. Clearly, if $G$ is the complete graph $K_n$ on vertex set $[n]$, then $G$ is $1^k$-representable by any permutation of $[n]$, in particular,  $w = 12\ldots n$ $1^k$-represents
$K_n$.

We proceed by induction on the number of edges in a graph with the base case being the complete graph. Our goal is to show that if $G$ is $1^k$-representable, then the graph $G'$ obtained from $G$ by removing any edge
$ij$ is also $1^k$-representable.

Suppose that $w$ $1^k$-represents $G=([n],E)$ and let $p(w)$ denote the {\em initial permutation} of $w$. That is, $p(w)$ is obtained from $w$ by removing all but the leftmost occurrence of each letter.
For example, 
$p(31443266275887)=31426758$.
Further, let $\pi$ be any permutation of $[n]\backslash\{i,j\}$.
Then we claim that the word
$$w'=i^{k-1}\pi ip(w)w$$
$1^k$-represents $G'$. Indeed, the vertices $i$ and $j$ are not connected any more because $w'_{\{i,j\}}$ contains $i^k$. Also, no new edge can be created because of the presence of $w$ as a subword. Thus, we only need to show that each edge $ms$ represented by $w$ is still represented by $w'$ if $s\neq j$ or $m\neq i$.

If $s \neq j$ and $m=i$, then either
\begin{enumerate}
\item $w_{\{s,i\}} = s^t i^d \ldots $ where $1 \leq d,t\leq k-1$  in which
case $w'_{\{s,i\}} = i^{k-1}sisis^t i^d \ldots$, or
\item $w_{\{s,i\}} = i^d s^t  \ldots $
where $1 \leq d,t \leq k-1$ in which
case $w'_{\{s,i\}} = i^{k-1}siisi^d s^t  \ldots $.
\end{enumerate}
In each case, it is easy to see that $w_{\{s,i\}}$ has no $1^k$-match so
that $w'$ $1^k$-represents the edge $is$.

If $m \neq i$ and $s=j$, then either
\begin{enumerate}
\item $w_{\{m,j\}} = m^t j^d \ldots $ where $1 \leq d,t\leq k-1$  in which
case $w'_{\{m,j\}} = mmjm^t j^d \ldots$, or
\item $w_{\{m,j\}} = j^d m^t  \ldots $
where $1 \leq d,t \leq k-1$ in which
case $w'_{\{m,j\}} = mjmj^d m^t  \ldots $.
\end{enumerate}
In each case, it is easy to see that $w_{\{m,j\}}$ has no $1^k$-match so
that $w'$ $1^k$-represents the edge $mj$.

Finally suppose  $m,s \not \in \{i,j\}$ and $m$ occurs before
$s$ in $\pi$. .
Then either
\begin{enumerate}
\item $w_{\{m,s\}} = m^t s^d \ldots $ where $1 \leq d,t\leq k-1$  in which
case $w'_{\{m,j\}} = msmsm^t s^d \ldots$, or
\item $w_{\{m,j\}} = s^d m^t  \ldots $
where $1 \leq d,t \leq k-1$ in which
case $w'_{\{m,j\}} = mssms^d m^t  \ldots $.
\end{enumerate}
In each case, it is easy to see that $w_{\{m,s\}}$ has no $1^k$-match so
that $w'$ $1^k$represents the edge $ms$.
\end{proof}

We note that there are some natural symmetries
among $u$-representable graphs. That
is, suppose that $u =u_1 \ldots u_j \in \P^*$ and
$\red(u) =u$.  Let the {\em reverse} of
$u$ be the word $u^r = u_j u_{j-1} \ldots u_1$.
Then for any word
$w \in \P^*$, it is easy to see that
$w$ has a $u$-match if and only if $w^r$ has a $u^r$-match.
This justifies the following
observation.

\begin{observation}\label{lem:reverse-u}
Let $G =(V,E)$ be a graph and  $u\in\P^*$ be such that $\red(u) =u$. Then
$G$ is $u$-representable if and only if $G$ is
$u^r$-representable.
\end{observation}

For any word
$w =w_1 \ldots w_k \in \P^*$ whose largest letter is $n$,
we let $w^c = (n+1-w_1) \ldots (n+1-w_k)$.
It is easy to see that
$w$ has a $u$-match if and only if $w^c$ has a $u^c$-match.
Given a graph $G =([n],E)$, we let
the {\em supplement} of $G$ be defined by
$\overline{G} =([n],\overline{E})$ where for all $x,y \in [n]$,
$xy \in E$ if and
only if $n+1-x$ and $n+1 -y$ are adjacent in $\overline{G}$.  One can think of
the supplement of the graph $G =(V,E)$ as a relabeling where
one replaces each label $x$ by the label
$n+1-x$.

It is easy to see that if $w$ witnesses that
$G=([n],E)$ is $u$-representable,
then $w^c$ witnesses that $\overline{G}$ is $u^c$-representable.
This justifies the following
observation.

\begin{observation}\label{lem:complement-u}
Let $G =([n],E)$ be a graph, and
$u$ be a word in $[n]^*$ such
that $\red(u) =u$. Then
$G$ is $u$-representable if and only if $\overline{G}$ is
$u^c$-representable.
\end{observation}

We can combine Observations 2 and 3 to prove
the following fact about 12-representable graphs. Suppose
that $w$ 12-represents $G$. Then $w^r$ 21-represents $G$ and, hence,
$(w^r)^c$ 12-represents $\overline{G}$.
It follows that if a vertex $v$ has label $1$ (resp., $n$) in some labeling realizing the $12$-representability of an unlabeled graph $G$,
then there is another labeling realizing the
$12$-representability of $G$ such that the vertex $v$ has label $n$ (resp., $1$).





\section{12-representable graphs}\label{12-matching-repr}

In this section we begin the study of $12$-representable
graphs.




Our first topic of study is the length of a word $w$ than
can 12-represent a graph.
Recall that  $G=([n],E)$ is a {\em permutation graph} if and only if
there is a permutation $\sg$ of $[n]$ such that for all
$1 \leq i < j \leq n$, $ij$ is in $E$ if and only
if $j$ occurs before $i$ in $\sg$.  However, this means
that $\sg$ $12$-represents $G$. Thus we have the
following simple fact.

\begin{proposition}\label{perm-graphs-versus-12-match} A graph $G$ can be
$12$-represented by a permutation if and only if $G$ is a permutation graph. \end{proposition}

It follows that all graphs on at most four vertices are $12$-representable (since $C_5$ is the minimum graph that is not a permutation graph). We will study the place of 12-representable graphs among the other graph classes in Section~\ref{graph-classes}.

Next we show that any 12-representable graph can be 12-represented by a word having at most two copies of each letter.

\begin{theorem}\label{thm-at-most-two} Let $G=(V,E)$ be a labeled representable graph.
Then there exists a word-representant $w$ in which each letter occurs at most twice. \end{theorem}

\begin{proof} Let $w'$ represent $G$ and suppose that a letter $j$ occurs in $w'$ more than twice. Then let $w'=AjBjC$, where $A$ and $C$ do not contain any copies of $j$. Note that $ij\in E$ if and only if ($i<j$ and all copies of $i$ are in $C$) or ($i>j$  and all copies of $i$ are in $A$). So, any copies of the letter $j$ in $B$ do not affect on the neighborhood of the vertex $j$ in $G$ and therefore they can be omitted. Doing the same with all other letters occurring in $w'$ more than twice, one obtain a required word $w$ representing $G$.
 \end{proof}

Note that replacing ``at most'' by ``exactly'' in the statement of Theorem~\ref{thm-at-most-two}, we obtain a true statement. This is based on the fact that replacing a letter $x$ in a word $12$-representing a graph by any number of copies of $x$, we obtain a word $12$-representing the same graph.





\begin{lemma}\label{copies-lemma} Let $G=(V,E)$ be a
$12$-representable graph and $v\in V$. Then the graph $H$ obtained by adding to $G$ a copy of $v$ is also
$12$-representable. \end{lemma}

\begin{proof} 
Let $i$ be the label of $v$. First, increase by 1 all labels $j>i$,
 keeping all other labels the same. Add a copy of $v$ and label it $i+1$ to obtain a labeling of $H$.
 Now, in a word $w$
 $12$-representing $G$ replace each letter $j>i$ by $j+1$ and substitute each occurrence of $i$ by $i(i+1)$ to obtain a word $w'$. Clearly, $ia$ is an edge in $H$ if and only if $(i+1)a$ is an edge in $H$, and $i$ and $i+1$ are not adjacent in $H$. All edges not involving $v$ and its copy are the same in $G$ and $H$. Thus, $w'$ $12$-represents $H$.\end{proof}



Next, we shall consider labeled graphs
$I_3$, $J_4$ and $Q_4$ presented in Figure \ref{fig:3graphs}.
These graphs will play a key role in determining which
graphs are
$12$-representable.

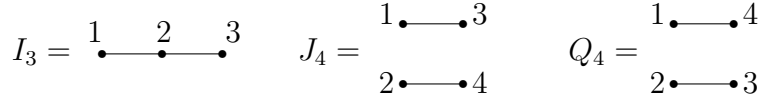
\begin{figure}[h]
\begin{center}

\setlength{\unitlength}{4mm}

\begin{picture}(20,3)

\put(0,0){

\put(-3,0.8){$I_3=$}

\put(-0.5,1.4){1} \put(1.8,1.4){2} \put(4.1,1.4){3}

\put(0,1){\p} \put(2,1){\p}\put(4,1){\p}

\put(0,1){\line(1,0){4}}
}

\put(10,0){

\put(-3.5,0.8){$J_4=$}

\put(-0.8,2){1} \put(2.3,2){3} \put(2.3,-0.3){4} \put(-0.8,-0.3){2}

\put(0,0){\p} \put(0,2){\p} \put(2,0){\p} \put(2,2){\p}

\put(0,0){\line(1,0){2}}
\put(0,2){\line(1,0){2}}
}

\put(19,0){

\put(-3.5,0.8){$Q_4=$}

\put(-0.8,2){1} \put(2.3,2){4} \put(2.3,-0.3){3} \put(-0.8,-0.3){2}

\put(0,0){\p} \put(0,2){\p} \put(2,0){\p} \put(2,2){\p}

\put(0,0){\line(1,0){2}}
\put(0,2){\line(1,0){2}}
}

\end{picture}
\caption{The graphs $I_3$, $J_4$, and $Q_4$.}\label{fig:3graphs}
\end{center}
\end{figure}

\begin{lemma}\label{thm:3graphs}
Let $G=(V,E)$ be a labeled graph.
Then if $G$ has an induced subgraph $H$ such that
$\red(H)$ is equal to one of $I_3$, $J_4$, or $Q_4$, then
$G$ is not $12$-representable.
\end{lemma}

\begin{proof}
First, suppose that $G$ has a subgraph $H$ such that
$\red(H) = I_3$.  Thus $H$ must be of the form
$H = (\{i,j,k\},\{ij,jk\})$ where $i < j < k$.
Now, for a contradiction,
suppose that $w=w_1 \ldots w_n$
 $12$-represents $G$.  Let $w_m$ be the left-most occurrence of $j$ in $w$. Then
since $ij \in E$, no
 $i$ occurs in $w_1 \ldots w_{m-1}$, and since
$jk \in E$, no $k$ occurs in $w_{m+1} \ldots w_{n}$.
But then, clearly, $w_{\{i,k\}}$ has no $12$-match,
which contradict the condition that $ik\not\in E$.

Next, suppose that $G$ has a subgraph $H$ such that
$\red(H) = J_4$ or $\red(H) = Q_4$.  Thus, $H$ must be of the form
$H = (\{i,j,k,\ell\},\{ik,j\ell\})$ where $\max\{i, j\} < \min\{ k, \ell\}$.
Again, for a contradiction,
suppose that $w=w_1 \ldots w_n$
$12$-represents
$G$.  Let $w_t$ be the right-most occurrence of $k$ in $w$. Then
since $ik \in E$, no $i$ occurs in $w_1 \ldots w_{t-1}$, and since
$jk \not \in E$, it must be the case
that $j$ occurs in $w_1 \ldots w_{t-1}$. Let
$w_s$ be the left-most occurrence of $j$ in $w$. Then $s < t$.  But since $j\ell  \in E$, no $\ell$ occurs in $w_{s+1} \ldots w_{n}$. Next,
let $w_r$ be the right-most occurrence of $\ell$ in $w$. Then
$r < s$. But this would imply that $i\ell \in E$
which is a contradiction.
\end{proof}

An immediate corollary to Lemma~\ref{thm:3graphs} is that in a
12-representable labeled graph, the labels must alternate in size through any induced path in the graph.

Let us say that a labeled graph $G=(V,E)$, where $V \subset \P$, has
a {\em bad path} if $G$ has an induced path $P$ whose endpoints are labeled by two smallest elements in $P$.


\begin{lemma}\label{thm:12paths}
Let $G=([n],E)$ be a labeled graph.
Then if $G$ has
a bad path $P$ of length at least $3$, then
$G$ is not $12$-representable.
\end{lemma}
\begin{proof}
Let $P=x_0x_1\ldots x_s$ and $\max \{x_0,x_s\} < \min \{ x_1,x_2,\ldots, x_{s-1}\}$. If $s\ge 4$
then the reduction of the subgraph induced by $\{x_0,x_1,x_{s-1},x_s\}$ is $J_4$ or $Q_4$. If $s=3$ then without loss of generality $x_1 < x_2$, and the reduction of the subgraph induced by $\{x_0,x_1,x_2\}$ is
$I_3$. In both cases $G$ is not 12-representable by Lemma~\ref{thm:3graphs}. \end{proof}


We note that Lemma~\ref{thm:12paths} does not say that paths
are not
 12-representable, it only states the certain properties of its labeling. In fact, all paths are 12-representable since they are caterpillars, and all caterpillars are permutation graphs \cite{GrCl}.





\begin{lemma}\label{glueing-by-edge} Let $G=(V_G,E_G)$ and $H=(V_H,E_H)$
be $12$-representable graphs.
 Assume that there are labelings of $G$ and
 $H$ such that $x\in V_G$ and $y\in V_H$ receive the smallest or the highest labels.
  Then the graph $G \cup H \cup \{xy\}$
   obtained from
disjoint copies of $G$ and $H$ by adding the edge $xy$, is
 $12$-representable.
 \end{lemma}

\begin{proof} Suppose, without loss of generality, that in our labelings $V_G=\{1,2,\ldots,k\}$ and  $V_H=\{k+1,k+2,\ldots,\ell \}$.
Moreover, by Lemma~\ref{lem:complement-u} we can assume that $x$ is labeled by $k$ and $y$ is labeled by $k+1$.
Denote by $w_G$ and $w_H$
the words $12$-representing $G$ and $H$,
respectively.
Let $w_G'$ be the word obtained from $w_G$ by replacing each occurrence of $k$ by $k+1$, and $w_H'$ be the word obtained from $w_H$ by replacing each occurrence of $k+1$ by $k$.
It is easy to see that the word $w=w_G'w_H'$
represents the graph $G \cup H \cup \{xy\}$.
\end{proof}

Given two subsets of positive integers $A$ and $B$,
we write $A < B$ if every element of $A$ is less than every element
of $B$, i.~e. $x < y$ whenever $x \in A$ and $y \in B$. A subset $U\subset V$ is a {\it cutset} if $G\setminus U$ is disconnected.

\begin{lemma}\label{thm:disjoint}
 Suppose that $G =([n],E)$ is a labeled graph.
 Let $U$ be a cutset of G.
Denote by $G_1 =(V_1,E_1)$ and
$G_2 = (V_2,E_2)$ two components of $G\setminus U$.
If $G$ is $12$-representable,
$|V_1|\geq 2,\ |V_2| \geq 2$, and the smallest element
of $V_1 \cup V_2$ lies in $V_1$, then $V_1 < V_2$.
\end{lemma}
\begin{proof}
Let $H=\red(G_1 \cup G_2)$. Then $1\in V_1$. Denote by $k>1$ the smallest element in $V_2$.
Assume that the property $V_1 < V_2$ does not hold. Then $V_1$ contains labels that are greater than $k$.
 Denote by $C_1$ (resp., $C_2$) the set of all vertices in $V_1$ whose labels are less (resp., greater) than $k$.
Since $G_1$ is connected, there is an edge $ab$ such that $a\in C_1$ and $b\in C_2$. Denote by $\ell$ a neighbor of $k$ in $G_2$. Then the reduction of the subgraph induced by $\{ a,b,k,\ell\}$ is either $J_4$ or $Q_4$, and so $G$ is not representable by Lemma~\ref{thm:3graphs}, a contradiction.
Hence, $V_1<V_2$.
\end{proof}

The following theorem provides examples of non-12-representable graphs.
Note that we have shown that $C_3$ and $C_4$ are 12-representable. It turns
out that they are the only 12-representable cycles.

\begin{theorem}\label{cycles-thm}  $C_n$ is not $12$-representable for
any $n \geq 5$.
\end{theorem}
\begin{proof}
Suppose for a contradiction that
$C_n$ is 12-representable where $n \geq 5$. Let $1,x_1, \ldots, x_{n-1}$ be the
labels of vertices as we proceed around the cycle in a clockwise order.
Then since no subgraph of $C_n$ can reduce to $I_3$ by
Lemma~\ref{thm:3graphs},
the sequence $1,x_1,x_2, \ldots, x_{n-1},1$ must be an up-down sequence, i.~e.
$1<x_1 >x_2 < x_3 > x_4 < \cdots x_{n-2} < x_{n-1} > 1$.  This is
clearly impossible if $n$ is odd.  Now assume
that $n$ is even. But then consider the position of 2
in the sequence $1,x_1,x_2, \ldots, x_{n-1},1$.  Clearly $2$ cannot
be equal to $x_1$ or $x_n$.  But this means that one of
the two paths that connect 1 to 2 around the cycle would
be a bad path of length at least 3 which is impossible
by Lemma~\ref{thm:12paths}.
\end{proof}

\section{Characterization of $12$-representable trees}\label{trees}

A {\em caterpillar}
is a tree in which all the vertices are within distance 1 of a central path. In this paper, we need the notion of a double caterpillar defined as follows.

\begin{definition} A {\em double caterpillar} $T$ is a tree in which all the vertices are within distance 2 of a central path. Such a path is called a double caterpillar's {\em spine} if it is obtained by first removing all leaves from $T$ and then removing all leaves from the obtained tree. \end{definition}

A {\em star} or {\em star tree}
is the {\em complete bipartite graph} $K_{1,n}$. Here we allow $n\geq 0$, where $n=0$ corresponds to the graph $K_1$ (an  isolated vertex). The {\em centrum} of  a star is the all-adjacent vertex in it. Suppose that a vertex $v$ in a tree $T$ is adjacent to vertices $v_1, v_2,\ldots,v_k$. Removing $v$
we obtain a forest $T\backslash v$ whose $i$th component $T_i$ is determined by the tree having $v_i$ as a vertex. We say that the $i$th component of the forest is {\em good} if it is a star with centrum at the vertex $v_i$.

\begin{lemma}\label{good-comp} If a tree $T$ is
$12$-representable then for any vertex $v$, at most two components $T_i$ of the forest $T\backslash v$ are not good. \end{lemma}

\begin{proof} Note that all trivial (one-vertex) components of $T\backslash v$ are good by the definition. Let $T_1,T_2, \ldots,T_k$ be non-trivial components of $T\backslash v$. By Lemma~\ref{thm:disjoint} we can assume that the labels of these components satisfy the property $T_1< T_2< \cdots< T_k$.

Now, suppose that there are three components of the forest $T\backslash v$
which are not good. Without loss of generality, we can assume that these components are $T_1$, $T_2$ and $T_3$. Further, assume that the vertices $v$, $v_1$, $v_2$ and $v_3$ receive labels $r$, $m_1<m_2<m_3$, respectively, in some labeling $T'$ realizing representability of $T$. Since $T_2$ is not good, it contains two vertices $y_1,y_2$ such that $y_1m_2,y_1y_2\in E$ and $y_2m_1\not\in E$. Note that a similar statement is true for $T_1$ and $T_3$.
The structure of these components is schematically shown in Figure~\ref{fig:3tree}.

Note that if $m_1<r<m_3$ then we obtain a contradiction with Lemma~\ref{thm:3graphs} since the reduction of $\{v, v_1,v_3\}$ induces $I_3$.

\fig{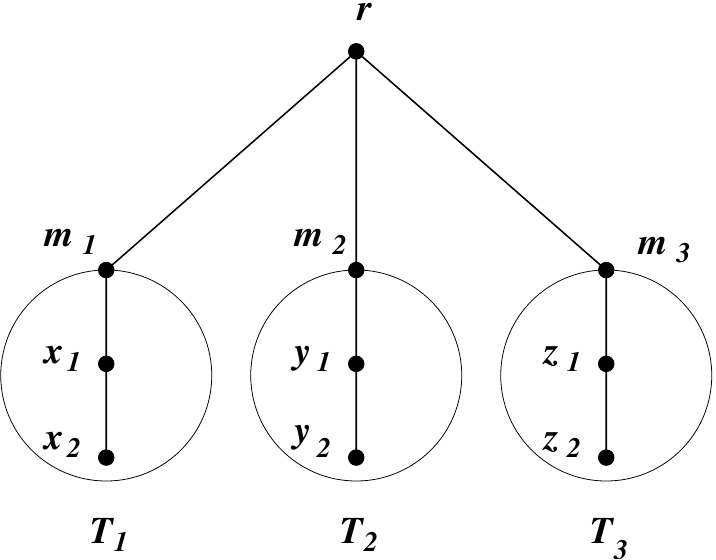}{The structure of components in $T\backslash v$ which are
not good.}

We can now assume that $r<m_1$, since for the case $r>m_3$ we can take the supplement of $T'$ and apply the observation about
12-representable graphs given at the end of Section 2.
Since $T_1<T_2<T_3$ we have $r<\min \{y_1,y_2\}<\max \{y_1,y_2\} <m_3$, and therefore the subgraph induced
by the vertices $\{r,y_1,y_2,m_3\}$ reduces to a copy of $Q_4$, which is impossible
by Lemma~\ref{thm:3graphs}.
\end{proof}

Note that a tree $T\backslash v$ can have two components which are
not good (see, for example, Figure~\ref{fig:F9}), and thus Lemma~\ref{good-comp} cannot be enhanced.

\fig{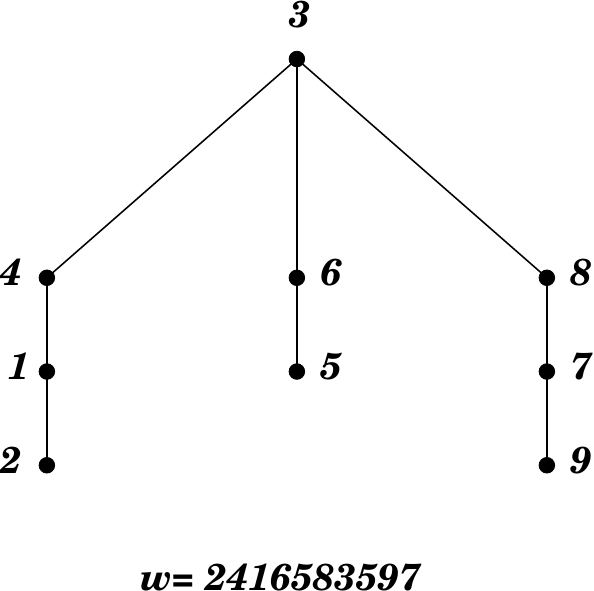}{Removal of $3$ produces a forest with two components which
are not good.}

The main result of this section is the following characterization of
12-representable trees.

\begin{theorem}\label{trees-characterization} A tree $T$ is
$12$-representable if and only if it is a double caterpillar.\end{theorem}

\begin{proof} {\bf Necessity.} Suppose that a tree $T$ is
not a double caterpillar. Further, suppose that $P=v_1v_2\ldots v_k$ is a longest path in $T$. Since all trees of diameter 5 are double caterpillars, $P$ has at least six edges, and thus $k\geq 7$. By our assumption, $T$ has a vertex $v$ at distance 3 from $P$. Suppose that $v_i$ is the closest to $v$ vertex on the path $P$. Since $P$ is of maximum length, we have $i\in\{4,5,\ldots,k-3\}$. But then in the forest $T\backslash v_i$ at least three components
which are not good, namely those containing $v$, $v_1$ and $v_k$. Thus by Lemma~\ref{good-comp}, $T$ is not $12$-representable.

\fig{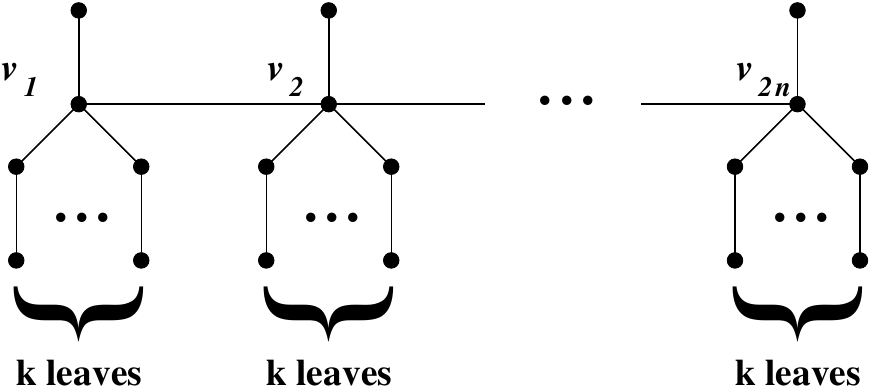}{A uniform double caterpillar with even spine.}

\noindent
{\bf Sufficiency.} By
Lemma~\ref{copies-lemma}, we can assume that no leaf has a sibling. To show that any such double caterpillar is
12-representable, we will use induction on the length of double caterpillar's spine, and prove the statement for {\em uniform double caterpillars} $DC(P_{2n})$ with even spines $P_{2n}=v_1v_2\ldots v_{2n}$ presented schematically in
Figure~\ref{fig:caterpillar};
then any other double caterpillar will be 12-representable due to Lemma~\ref{copies-lemma} and Observation~\ref{subgraph}.

\fig{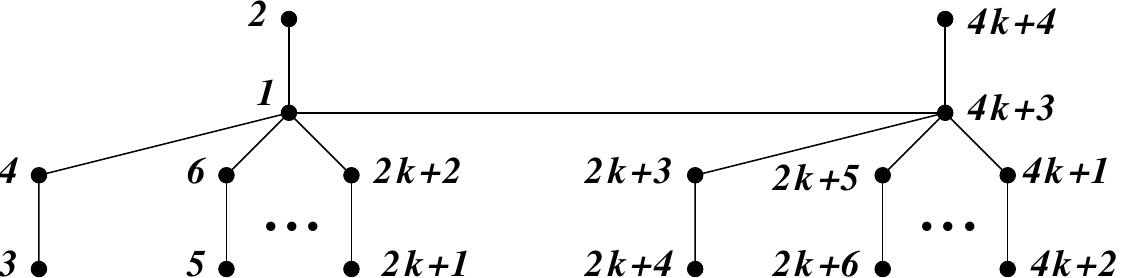}{The labeling of $DC(P_2)$.}

We will prove even a stronger statement, namely that there is a labeling of $DC(P_{2n})$ in which the label of $v_1$ is 1 and that of $v_{2n}$ is the maximum label $2n(k+1)$. The base of the induction is given by labeling $DC(P_2)$ presented in Figure~\ref{fig:caterpillar-base}, and
the following 12-representant:
$$24365\ldots(2k+2)(2k+1)(2k+4)(2k+6)\ldots (4k+2)(4k+4)135\ldots$$
\vspace{-20pt}
$$(2k+1)(2k+4)(2k+3)(2k+6)(2k+5)\ldots (4k+2)(4k+1)(4k+3)$$
stated on two lines. It is straightforward to check that this word has the right alternating properties.

Now, suppose that we are given a double caterpillar $DC(P_{2n})$. Choose any $1\leq r\leq n-1$  and remove the edge $v_{2r}v_{2r+1}$ on $DC(P_{2n})$'s spine. We get two double caterpillars with even spines $DC(P_{2r})$ and  $DC(P_{2(n-r)})$ on $s=2r(k+1)$ and $t=2(n-r)(k+1)$ vertices respectively. We can now apply the induction hypothesis to $DC(P_{2r})$ and  $DC(P_{2(n-r)})$, i.~e. consider the labeling of $DC(P_{2r})$ where $v_1$ has the smallest label $1$ and $v_{2r}$ has the largest label $s$, and
the labeling of $DC(P_{2(n-r)})$ where $v_{2r+1}$ has the smallest label $s+1$ and $v_{2n}$ has the largest label $s+t$.
Now apply Lemma~\ref{glueing-by-edge} to connect these graphs by the edge $v_{2r}v_{2r+1}$ thus obtaining a
labeling realizing 12-representability of $DC(P_{2n})$ in such a way that $v_1$ has the smallest label $1$ and $v_{2n}$ has the largest label $s+t=2n(k+1)$ (recall that in the proof of  Lemma~\ref{glueing-by-edge} no vertices except for the endpoints of the inserted edge changed their labels).
 \end{proof}

\fig{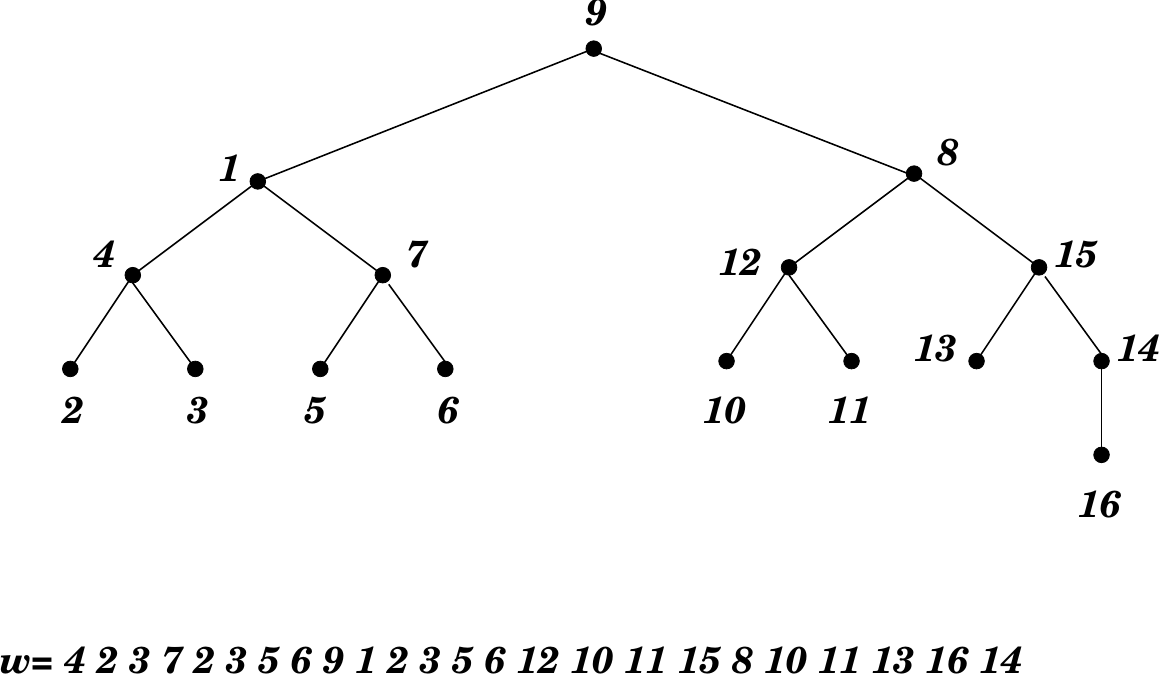}{A 12-representation of the full binary tree of
height 3 plus one vertex.}

Note that the labeling presented in the proof of Theorem~\ref{trees-characterization} is not the only possible labeling for double caterpillars. For example, the tree presented in Figure~\ref{fig:FullBin} has the spine $1,9,8,15$, while the maximum label is $16$.

%

\section{12-representable graphs and known classes of graphs}\label{graph-classes}

The goal of this section is to justify Figure~\ref{overview}.

Let us recall the definitions and known properties of some graph classes.
A {\em comparability graph} is an undirected graph that connects pairs of elements that are comparable to each other in a partial order (a poset). Comparability graphs are also known as {\em transitively orientable graphs} or {\em partially orderable graphs}. A {\em transitive orientation} of a graph is an acyclic orientation that has a property that if $a\rightarrow b$ and $b\rightarrow c$  are arcs then we must have the arc $a\rightarrow c$.  A graph $G$ is a {\em co-comparability graph} if its complement $G^c$ is a comparability graph.
 It is known \cite{DM} that a graph $G$ is a permutation graph if and only if both $G$ and its complement $G^c$ are comparability graphs. An {\em interval graph} is the intersection graph of a family of intervals on the real line. It has one vertex for each interval in the family, and an edge between every pair of vertices corresponding to intervals that intersect. A graph $G$ is {\em co-interval} if its complement $G^c$ is an interval graph.
A graph is {\em chordal} if it has no induced cycle on at least 4 vertices. It is a well known fact \cite{GH} that a graph is an interval graph if and only if it is chordal and a co-comparability graph.

As it is mentioned in the introduction, any comparability graph is word-representable \cite{KS}, and any odd cycle of length 5 or more, being a
non-comparability graph, is word-representable \cite{KL}. Moreover, odd wheels on six or more vertices are non-word-representable~\cite{KP2008}, and the set of $1^k$-representable graphs, for any $k\geq 3$, coincides with the set of all graphs by Theorem~\ref{11111-matching}. Our next result shows that any $12$-representable graph is necessarily a comparability graph.

\begin{theorem} If $G$ is a $12$-representable graph, then $G$ is a comparability graph.\end{theorem}

\begin{proof}
By Lemma~\ref{thm:3graphs}, any induced path $P$ of length 3 is such that $\red(P)\neq I_3$. We now direct edges in $G$ so that if $ab$ is an edge and $a<b$ then the arc $a\rightarrow b$ goes from $a$ to $b$. This orientation is obviously acyclic. We claim that this orientation is, in fact, transitive, which completes the proof of our theorem. Indeed, if the directed copy of $G$ contains a directed path $\vec{P}$ of length 3, say $a\rightarrow b\rightarrow c$, then we must have the arc $a\rightarrow c$ in the graph or otherwise $\red(\vec{P})=I_3$.    \end{proof}




\begin{figure}[ht]
\begin{center}
\includegraphics[scale=0.9]{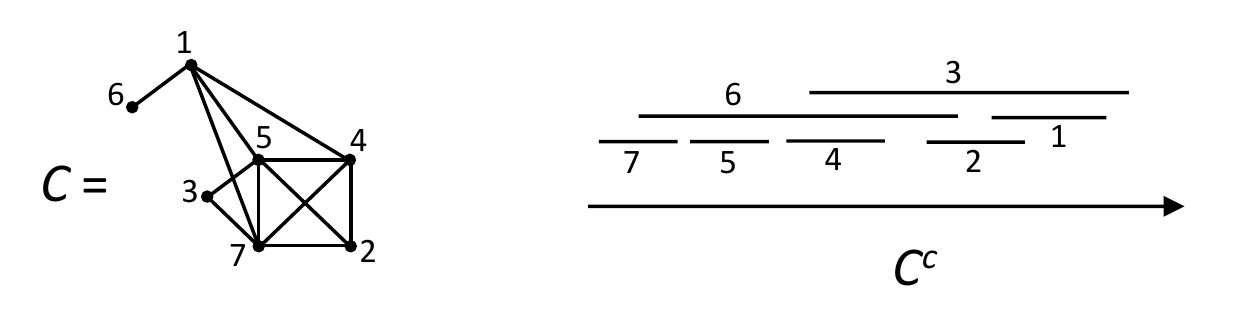}
\end{center}
\vspace{-10pt}
\caption{A co-interval graph $C$ and an interval representation of its complement $C^c$.}
\label{co-interval}
\end{figure}

\begin{theorem} If $G$ is a co-interval graph, then $G$ is $12$-representable.\end{theorem}

\begin{proof} Suppose that $G$ is a co-interval graph on $n$ vertices. It is a well-known easy fact that for any interval graph, there is its interval representation such that the endpoints of intervals are all distinct. Consider such an interval representation of the complement graph $G^c$. Next, put to an interval in this representation a label $n-i+1$ if the left endpoint of this interval is the $i$th one from left to right among all left endpoints. Such a labeling induces a labeling of $G$. We refer to Figure~\ref{co-interval} for an example of a co-interval graph $C$ and its labeling based on the endpoints of the intervals.

Next, form a word $w$ corresponding to labeled intervals by going through all interval endpoints (both left and right endpoints) from left to right and recording their labels in the order we meet them. For example, for the labeled interval representation in Figure~\ref{co-interval}, the word $w$ is $76755434261213$. Optionally, all occurrences of $ii$, like $55$ in the last word, can be replaced by a single $i$. We claim that the word $w$ $12$-represents $G$. Indeed, let $i<j$. If $i$th and $j$th intervals overlap, then $w_{\{i,j\}}=jiji$ or $w_{\{i,j\}}=jiij$; anyway, $i$ and $j$ are not adjacent. Otherwise,
by the choice of the labeling, the $i$th interval lies directly to the right from the $j$th one, and thus $w_{\{i,j\}}=jjii$, i.~e. $ij$
is an edge.
\end{proof}

To conclude our description of Figure~\ref{overview}, we would like to justify that the Venn diagram presented by us is proper, namely that there are strict inclusions of sets and also there is no inclusion of the class of co-interval graphs into the class of permutations graphs, and vice versa, and these classes do overlap. Note that it remains to explain the set inclusions only inside the class of $12$-representable graphs since the rest of the diagram has been already explained above.

\begin{figure}[ht]
\begin{center}
\includegraphics[scale=0.9]{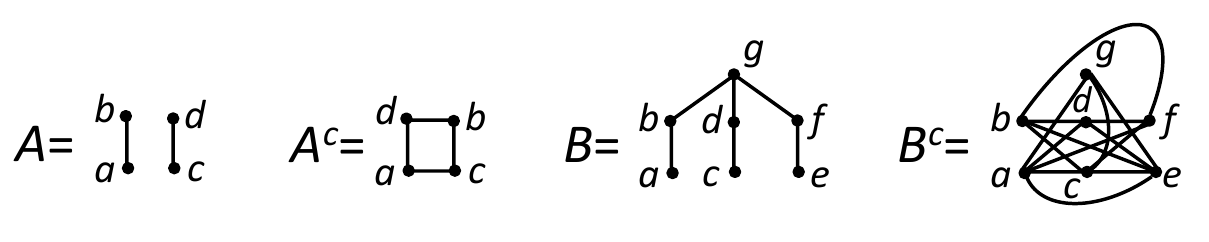}
\end{center}
\vspace{-10pt}
\caption{Graphs $A$ and $B$ and their complements $A^c$ and $B^c$.}
\label{four-graphs}
\end{figure}

Clearly, complete graphs
are both co-interval graphs (for the set of non-intersecting intervals) and permutation graphs (for the identity permutation).
In Figure~\ref{four-graphs}, there are two graphs, $A$ and $B$, and their complements $A^c$ and $B^c$.
The graph $A$ is a permutation graph (for 2143) but not a co-interval graph, because its complement is not chordal.
The graph $B$ is $12$-representable by Theorem~\ref{trees-characterization}, while it is neither a permutation graph nor a co-interval graph since $B^c$ is neither a comparability graph \cite{GrCl} (note that $B=T_2$ in their notation) nor a chordal graph ($adbc$ in an induced $C_4$).

Finally, for the sake of completeness, let us provide an example of a co-interval graph that is not a permutation graph.
Consider the graph $G_n$ whose vertices are defined by all intervals of non-zero length with left endpoints in the set $\{0,1,\ldots,n\}$ and right endpoints in the set $\{1-\epsilon,2-\epsilon,\ldots,n-\epsilon\}$, where $\epsilon\in (0,1)$. Further, two vertices are connected in $G_n$ by an edge if and only if the intervals corresponding to them do not overlap. By definition, $G_n$ is a co-interval graph. It is therefore a comparability graph corresponding to the following poset $P$ on $V(G_n)$: $I<J$ if and only if the interval $I$ lies entirely to the left of the interval $J$.
We claim that $G_n$ is not a permutation graph if $n$ is large enough.
This follows from two known facts. First \cite{BFR}, a graph G is a permutation graph if and only if it is the comparability graph of a poset that has dimension at most $2$. On the other hand, the Example 8.1.4 in \cite{S} shows that the dimension of the poset $P$ grows arbitrary large while increasing $n$. Therefore, for large enough $n$, the graph $G_n$ becomes a non-permutation graph.

\section{Grid graphs}\label{conclusion}

In this section, we consider certain induced subgraphs of a {\em grid graph} or {\em polyominoes}.  Examples of a grid graph and some of its possible induced subgraphs are given in Figure~\ref{grid-graphs}, where the notions of  ``{\em corner graphs}'' and ``{\em skew ladder graphs}'' were invented by us.


\begin{figure}[ht]
\begin{center}
\includegraphics[scale=0.8]{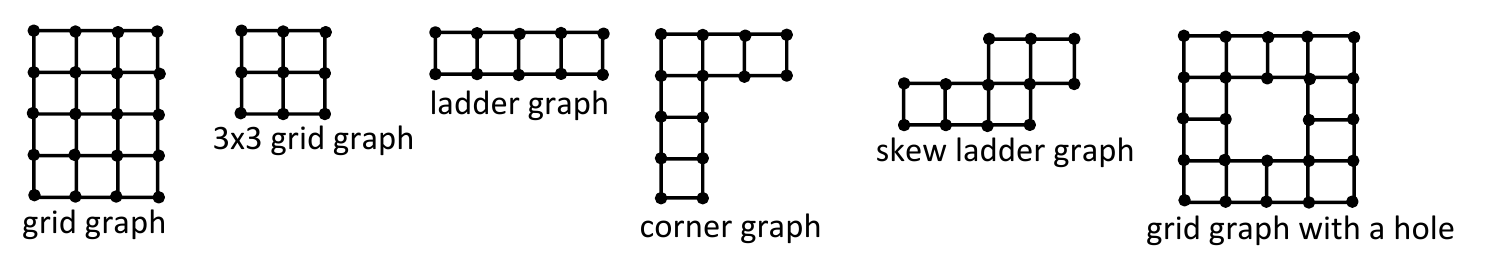}
\end{center}
\vspace{-10pt}
\caption{Induced subgraphs of a grid graph.}
\label{grid-graphs}
\end{figure}


Clearly, grid graphs with holes or grid graphs containing a $3\times 3$ grid subgraph
are not $12$-representable because of large induced cycles (cycles of length at least 8) contained in them,
which are not possible in $12$-representable graphs by Theorem~\ref{cycles-thm}.



\begin{figure}[ht]
\begin{center}
\includegraphics[scale=0.8]{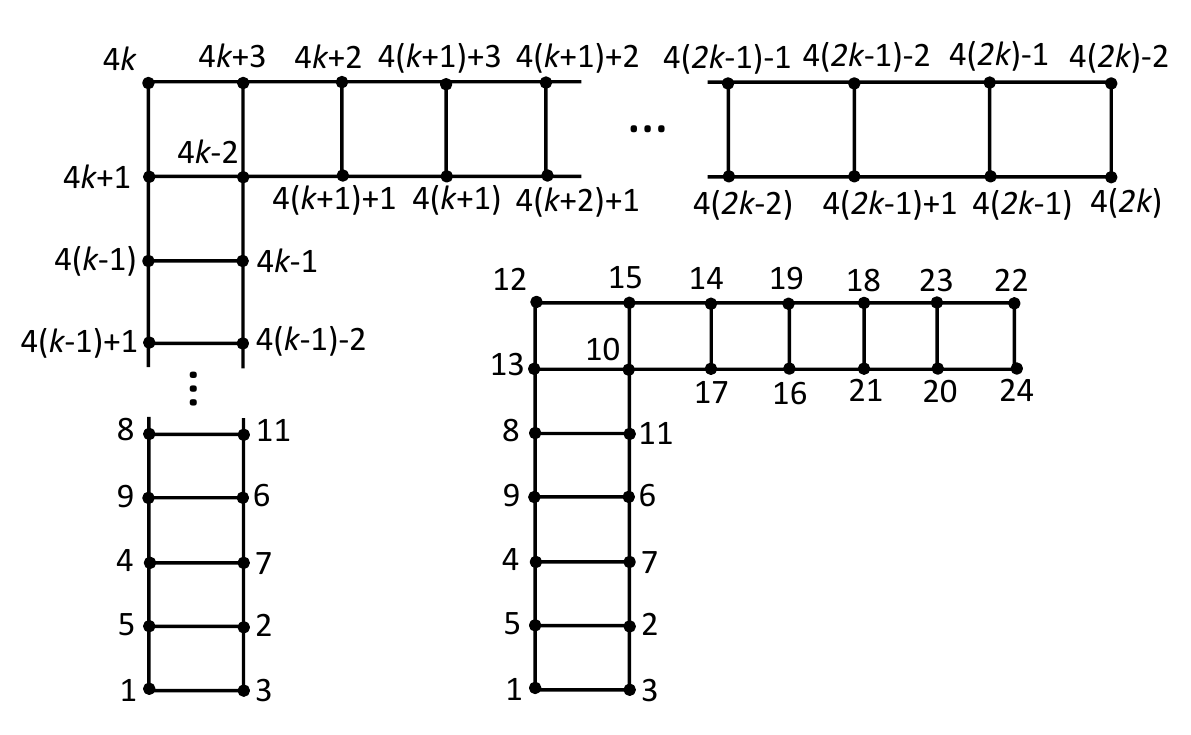}
\end{center}
\vspace{-10pt}
\caption{Labeling corner graphs to show their $12$-representability.}
\label{corner}
\end{figure}

The situation with ladder graphs, corner graphs and skew ladder graphs is different. These graphs turn out to be $12$-representable. Note that such a representability for ladder graphs follows from representability of any of the other two classes of graphs.

To  show that corner graphs are $12$-representable, one can consider labelling as shown in Figure~\ref{corner} in general case, and in case of $k=3$ to help the reader to follow the labelling. Words, $12$-representing the general and particular cases, respectively, are as follows
\begin{tiny}
$$3.51.72.94.(11)6.\cdots .(4k+1)(4(k-1)).(4k+3)(4k).(4(k+1)+1)(4k-2).{\bf 4k}.$$
\vspace{-10pt}
$$(4(k+1)+3)(4k+2).
(4(k+2)+1)(4(k+1)).\cdots.(4(2k)-1)(4(k+1)+2).(4(2k))(4(2k-1)).(4(2k)-2)$$
\end{tiny}
and
\begin{tiny}
$$3.51.72.94.(11)6.(13)8.(15)(12).(17)(10).{\bf (12)}.(19)(14).(21)(16)(23)(18).(24)(20).(22),$$
\end{tiny}

\noindent
where the dots just help seeing the patterns in our construction, and the first word is on two lines. Note the corner element in bold that is repeated in our construction. We do not provide a careful justification of why these words work, which can be seen by direct inspection.

\begin{figure}[ht]
\begin{center}
\includegraphics[scale=0.9]{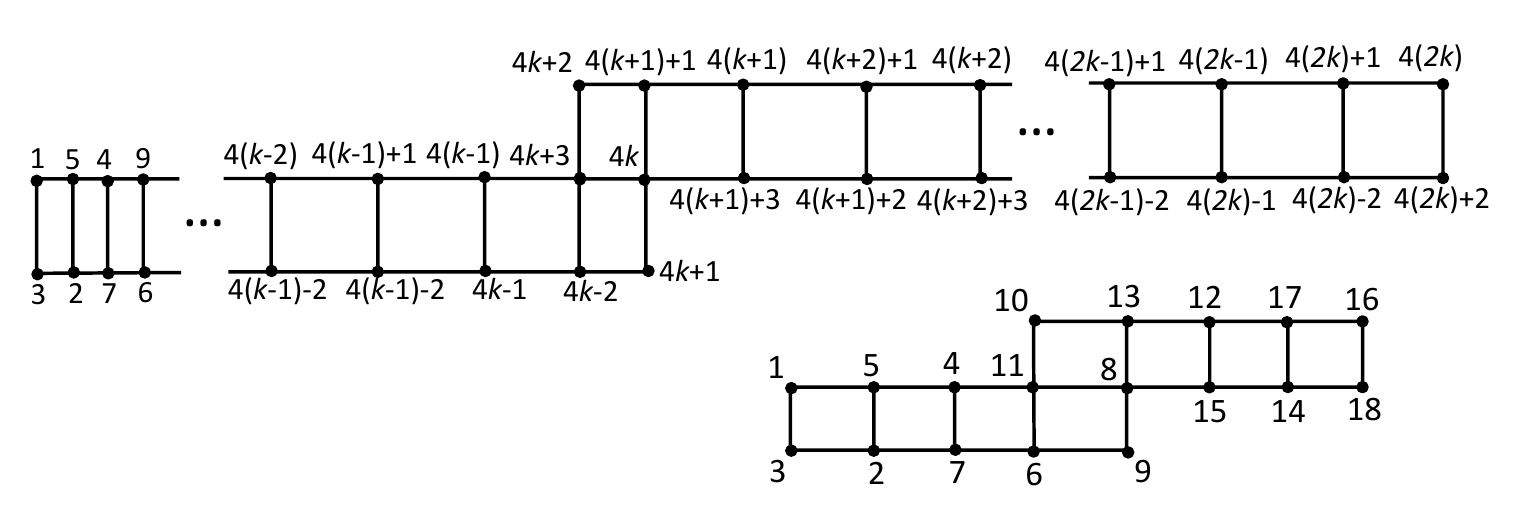}
\end{center}
\vspace{-10pt}
\caption{Labeling skew ladder graphs to show their $12$-representability.}
\label{skew-ladder}
\end{figure}

To  show that skew ladder graphs are $12$-representable, one can consider labelling as shown in Figure~\ref{skew-ladder} in general case, and in case of $k=2$ to help the reader to follow the labelling. Words, $12$-representing the general and particular cases, respectively, are as follows
\begin{tiny}
$$3.51.72.94.(11)6.\cdots.(4(k-1)+1)(4(k-2)).(4k-1)(4(k-1)-2).{\bf (4k+1)}.(4k+3)(4(k-1)).$$
\vspace{-10pt}
$$(4k+1)(4k-2).(4(k+1)+1)(4k+2).(4(k+1)+3)(4k).{\bf (4k+2)}.(4(k+2)+1)(4(k+1)).$$
\vspace{-10pt}
$$(4(k+2)+3)(4(k+1)+2).\cdots.(4(2k)+1)(4(2k-1)).(4(2k)+2)(4(2k-1)+2).(4(2k))$$
\end{tiny}
and
\begin{tiny}
$$3.51.72.{\bf 9}.(11)4.96.(13)(10).(15)8.{\bf 10}.(17)(12).(18)(14).(16),$$
\end{tiny}

\noindent
where the first word is on three line, and again, in bold we indicate repeated corner elements.

It would be interesting to know whether or not induced subgraphs of a grid graph have a nice $12$-representation classification, which we leave as an open problem along with the larger problem of finding a classification of $12$-representable graphs.

\section{Other notions of word-representable graphs}\label{other-notions}

As it is mentioned in Section~\ref{preliminaries}, apart from our main generalization, given in Definition~\ref{def-match-repr}, of the notion of a word-representable graph, we have another generalization given in Definition~\ref{exact-match-repr} below. In this section, we also state some other ways to define the notion of a (directed) graph representable by words. Our definitions can be generalized to the case of hypergraphs by simply allowing words defining edges/non-edges be over alphabets containing more than two letters.  However, the focus of this paper was studying $12$-representable graphs, so we leave all the notions introduced below for a later day to study.

Given a word $u =u_1 \ldots u_j \in \P^*$ such
that $\red(u) =u$, and a word $w = w_1 \ldots w_n \in \P^*$,
we say that the pattern $u$ {\em occurs} in $w$ if there exist
$1\leq i_1 < \cdots < i_j \leq n$ such that
$\red(w_{i_1} \ldots w_{i_j}) = u$, and that $w$ {\em avoids}
$u$ if $u$ does not occur in $w$.

Given a word $v =v_1 \ldots v_j \in \P^*$
and a word $w = w_1 \ldots w_n \in \P^*$,
we say that $v$ {\em exactly occurs} in $w$ if there exist
$1\leq i_1 < \cdots < i_j \leq n$ such that
$w_{i_1} \ldots w_{i_j} = v$ and that $w$ {\em exactly avoids}
$v$ if $v$ does not exactly occur in $w$.  We say
that $w$ has an {\em exact $v$-match starting at position $i$} if
$w_i w_{i+1} \ldots w_{i+j-1} =v$.

Similar definitions can
be made for set of words.  That is, let
$\Gamma$ be a set of words in $\P^*$ such that $\red(u) =u$ for
all $u \in \Gamma$.
Then we say that $\Gamma$ {\em occurs} in $w=w_1 \ldots w_n \in \P^*$
if there exist
$1\leq i_1 < \cdots < i_j \leq n$ such that
$\red(w_{i_1} \ldots w_{i_j}) \in \Gamma$, and that $w$ {\em avoids}
$\Gamma$ if $\Gamma$ does not occur in $w$.  We say
that $w$ has a {\em $\Gamma$-match starting at position $i$} if
$\red(w_i w_{i+1} \ldots w_{i+j-1}) \in \Gamma$.
Similarly, if $\Delta$ is any set of words in $\P^*$,
we say that $\Delta$ {\em exactly occurs} in $w=w_1 \ldots w_n \in \P^*$
if there exist
$1\leq i_1 < \cdots < i_j \leq n$ such that
$w_{i_1} \ldots w_{i_j} \in \Delta$, and that $w$ {\em exactly avoids}
$\Delta$ if $\Delta$ does not occur in $w$.  We say
that $w$ has an {\em exact $\Delta$-match starting at position $i$} if
$w_i w_{i+1} \ldots w_{i+j-1} \in \Delta$.

The study of pattern avoidance and pattern containment in words and permutations is a fast growing area (see \cite{Kit} for a comprehensive introduction to the field).

We defined the notion of a $u$-representable graph in Definition~\ref{def-match-repr}. More generally, we can make the same definition for
sets of words.

\begin{definition} Let $\Gamma$ be a set of words in $\{1,2\}^*$ such
that $\red(u) =u$ for all $u \in \Gamma$.  Then we say that a graph
$G =(V,E)$, where $V \subset \P$, is {\em $\Gamma$-representable}
if there exists a word $w \in \P^*$ such that $A(w) = V$ and
for all $x,y \in V$, $xy \not \in E$ if and only if
$w_{\{x,y\}}$ has a $\Gamma$-match.
\end{definition}

\begin{definition} Let $\Gamma$ be a set of words in $\{1,2\}^*$ such
that $\red(u) =u$ for all $u \in \Gamma$.  Then we say that a graph
$G =(V,E)$, where $V \subset \P$, is {\em $\Gamma$-occurrence representable}
if there exists a word $w \in \P^*$ such that $A(w) = V$ and
for all $x,y \in V$, $xy \not \in E$ if and only if
$\Gamma$ occurs in $w_{\{x,y\}}$.
\end{definition}

In the case where $\Gamma =\{u\}$ consists of a single word, we
simply say that a graph $G$ is $u$-occurrence representable
if $G$ is $\Gamma$-occurrence representable.  For example,
the $11$-occurrence representable graphs are very simple.
That is, if a word $w= w_1 \ldots w_n$ 11-occurrence represents
a graph $G =(V,E)$, then any vertex $x$ such that
$w$ has two or more occurrences of $x$, cannot be connected to
any other vertex $y$ since 11 will always occur in $w_{\{x,y\}}$.
Let $I = \{x \in V: x \ \mbox{occurs more than once in $w$}\}$ and
$J = \{y \in V: y \ \mbox{occurs exactly once in $w$}\}$.
Then it is easy to see that the elements of $J$ must form a clique
in $G$, while the elements of $I$ form an independent set.  Thus, if $G$ is 11-occurrence representable, then $G$ consists
of a clique together with a set of isolated vertices. Clearly,
all  such graphs are 11-occurrence representable, which gives a characterisation of 11-occurrence representable graphs.

Another simple observation is that the sets of 12-representable graphs and 12-occurrences representable graphs coincide, since a word contains a 12-match if and only if it contains a 12-occurrence.

Similarly, we have the following analogues of our definition
for exact matchings and exact occurrences.

\begin{definition}\label{exact-match-repr} Let $\Delta$ be a set of words in $\P^*$.
Then we say that a graph
$G =(V,E)$, where $V \subset \P$, is
{\em exact-$\Delta$-representable}
if there is a word $w \in \P^*$ such that $A(w) = V$ and
for all $x,y \in V$, $xy \not \in E$ if and only if
$w_{\{x,y\}}$ has an exact $\Delta$-match.
\end{definition}

\begin{definition}\label{exact-occur-repr} Let $\Delta$ be a set of words in $\P^*$.
Then we say that a graph
$G =(V,E)$, where $V \subset \P$, is
{\em exact-$\Delta$-occurrence representable}
if there is a word $w \in \P^*$ such that $A(w) =  V$ and
for all $x,y \in V$, $xy \not \in E$ if and only if
$\Delta$ exactly occurs in $w_{\{x,y\}}$.
\end{definition}

Note that to avoid trivialities, while dealing with exact matchings or occurrences, the sets of words defining (non-)edges should be large and hopefully contain at least one word for each pair of vertices in $V$.
Clearly, the properties of (exact) $\Gamma$-representability and (exact) $\Delta$-occurrence representability
are hereditary.

%




Recall the definitions of the reverse $u^r$ and the complement $u^c$ in Section~\ref{preliminaries}. If $\Gamma$ is a set of words in $\P^*$, then
we let $\Gamma^r = \{u^r: u \in \Gamma\}$.
If $\Delta$ is a set of words in $u \in \{1, \ldots, n\}^*$ such
that $A(u) = \{1, \ldots, n\}$, then we let
$\Delta^c = \{u^c:u \in \Delta\}$. Then we have the following
observation generalizing and extending Observation~\ref{lem:reverse-u}.

\begin{observation}\label{lem:reverse}
Let $G =(V,E)$ be a graph, $\Gamma$ be a set
of words in  $\P^*$ such that $\red(u) =u$ for all $u \in \Gamma$. Then
\begin{enumerate}
 \item $G$ is $\Gamma$-representable if and only if $G$ is
$\Gamma^r$-representable.

\item $G$ is $\Gamma$-occurrence representable if and only if $G$ is
$\Gamma^r$-occurrence representable.
\end{enumerate}
\end{observation}


Recall the definition of the supplement $\overline{G}$ of a graph $G$ given in Section~\ref{preliminaries}.
The following observation generalizes and extends Observation~\ref{lem:complement-u}.

\begin{observation}\label{lem:complement}
Let $G =(V,E)$ be a graph, and
$\Delta$ be a set of words in $\{1, \ldots, n\}^*$ such
that $A(u) = \{1, \ldots, n\}$ for all $u \in \Delta$. Then
\begin{enumerate}
 \item $G$ is $\Delta$-representable if and only if $\overline{G}$ is
$\Delta^c$-representable.

\item $G$ is $\Delta$-occurrence representable if and only if $\overline{G}$ is
$\Delta^c$-occurrence representable.
\end{enumerate}
\end{observation}

Given two words $u,v \in \P^*$, we say that
$u$ and $v$ are {\em matching-representation Wilf-equivalent} (resp., {\em occurrence-representation Wilf-equivalent})
if for any graph $G$, a
labeling of $G$ that is $u$-matching (resp. $u$-occurrence)
representable exists if and only if a labeling
of $G$ that is $v$-matching (resp., $v$-occurrence)
representable exists.
Note, that Observations~\ref{lem:reverse} and~\ref{lem:complement} show that the matching-representation
and occurrence-representation Wilf-equivalence classes
are closed under reversal and complement.

Our notion of using patterns to represent graphs can
also be extended to give us a notion of representing
directed graphs via words.  That is, suppose that
we are given a directed graph $G=(V,E)$, where $E \subset V \times V$
and we are given two sets of words $\Gamma,\Delta$ in $\P^*$ such that
$\red(u) =u$ for all $u \in \Gamma$ and $\red(v) =v$ for
all $v \in \Delta$.

\begin{definition} We say that a directed graph
$G =(V,E)$, where $V \subset \P$, is
{\em $\Gamma,\Delta$-representable}
if there is a word $w \in \P^*$ such that $A(w) = V$ and
for all pairs $x<y$ in $V$, $(x,y) \not \in E$ if and only if
$w_{\{x,y\}}$ has a $\Gamma$-match and
$(y,x) \not \in E$ if and only if
$w_{\{x,y\}}$ has a $\Delta$-match.
\end{definition}

\begin{definition} We say that a graph
$G =(V,E)$, where $V \subset \P$, is {\em $\Gamma,\Delta$-occurrence representable}
if there is a word $w \in \P^*$ such that $A(w) = V$ and
for all pairs $x<y$ in $V$, $(x,y) \not \in E$ if and only if
$\Gamma$ occurs in $w_{\{x,y\}}$  and
$(y,x) \not \in E$ if and only if
$\Delta$ occurs in $w_{\{x,y\}}$.
\end{definition}

We can make similar definitions for exact matching and exact
occurrences. That is, let $\Gamma$ and $\Delta$ be two
sets of words in $\P^*$.

\begin{definition}
We say that a directed graph
$G =(V,E)$, where $V \subset \P$, is {\em exact $\Gamma,\Delta$-representable}
if there is a word $w \in \P^*$ such that $A(w) = V$ and
for all pairs $x<y$ in $V$, $(x,y) \not \in E$ if and only if
$w_{\{x,y\}}$ has an exact $\Gamma$-match and
$(y,x) \not \in E$ if and only if
$w_{\{x,y\}}$ has an exact $\Delta$-match.
\end{definition}

\begin{definition} We say that a graph
$G =(V,E)$, where $V \subset \P$, is
{\em exact-$\Gamma,\Delta$-occurrence representable}
if there is a word $w \in \P^*$ such that $A(w) = V$ and
for all pairs $x<y$ in $V$, $(x,y) \not \in E$ if and only if
$\Gamma$ exactly occurs in $w_{\{x,y\}}$  and
$(y,x) \not \in E$ if and only if
$\Delta$ exactly occurs in $w_{\{x,y\}}$.
\end{definition}

We can obtain other notions of word-representability by
mixing $\Gamma$-matches, exact $\Gamma$-matches, $\Gamma$-occurrences, and
exact $\Gamma$-occurrences with $\Delta$-matches,
exact $\Delta$-matches, $\Delta$-occurrences, and
exact $\Delta$-occurrences in the definitions above.


\end{document}